\documentclass[11pt,twoside]{amsart}
\usepackage[dvips]{graphicx}
\usepackage{amssymb}
\usepackage{latexsym}
\usepackage{xypic}

\newtheorem{theorem}{Theorem}[section]
\newtheorem{lemma}[theorem]{Lemma}
\newtheorem{proposition}[theorem]{Proposition}

\theoremstyle{definition}
\newtheorem{definition}[theorem]{Definition}
\newtheorem{example}[theorem]{Example}

\theoremstyle{remark}
\newtheorem{remark}[theorem]{Remark}
\newenvironment{definition-proposition}{\begin{def-prop} \em}{\end{def-prop}}

\def\P{\mathbb P_{\mathbb C}}
\def\H{\mathbb H}

\numberwithin{equation}{section}

\newcommand{\Flecha}[1]{\overset{\longleftrightarrow}{#1}}
\newcommand{\C}{\mathbb{C}}

\begin{document}

\title{  \sc{ The limit set of discrete subgroups of $PSL(3,\C)$ }}
\author{W. Barrera, A. Cano \& J. P. Navarrete }
\address{
Waldemar Barrera: Universidad Aut\'onoma de Yucat\'an Facultad de Matem\'aticas, Anillo Perif\'erico
Norte Tablaje Cat 13615 Chuburn\'a Hidalgo, M\'erida Yucat\'an. M\'exico.\\
Angel Cano:   Instituto de Matem\'atica Pura e Aplicada (IMPA), Rio de Janeiro, Brazil.\\
Juan Pablo Navarrete:  Universidad Aut\'onoma de Yucat\'an Facultad de Matem\'aticas, Anillo Perif\'erico
Norte Tablaje Cat 13615 Chuburn\'a Hidalgo, M\'erida Yucat\'an. M\'exico. }
\email{bvargas@uady.mx, angel@impa.br, jp.navarrete@uady.mx}

\thanks{Research partially supported by grants from CNPq}




\maketitle
\begin{abstract}
If $\Gamma$ is a discrete subgroup of $PSL(3,\C)$, it is determined
the equicontinuity region $Eq(\Gamma)$ of the natural action of
$\Gamma$ on $\P^2$. It is also proved that the action restricted to
$Eq(\Gamma)$ is discontinuous, and $Eq(\Gamma)$ agrees with the
discontinuity set in the sense of Kulkarni whenever the limit set of
$\Gamma$ in the sense of Kulkarni, $\Lambda(\Gamma)$, contains at
least three lines in general position. Under some additional
hypothesis, it turns out to be the largest open set on which
$\Gamma$ acts discontinuously. Moreover, if $\Lambda(\Gamma)$
contains at least four complex lines and $\Gamma$ acts on $\P^2$
without fixed points nor invariant lines, then each connected
component of $Eq(\Gamma)$ is a holomorphy domain and a complete
Kobayashi hyperbolic space.
\end{abstract}

\section*{Introduction}\label{sec:1}

The study of the dynamics of discrete groups of automorphisms of
$\P^2$ is in its childhood, and one of the first basic problems is
understanding how the various possible notions of ``limit set"
relate among themselves. This is the main topic of this article.

If $\Gamma \subset PSL(2,\C)$ is a classical Kleinian group, then
its limit set $\Lambda$ is defined as the set of accumulation points
of the $\Gamma$ orbits in $\P^1 \cong S^2$. Well known arguments
show  that $\Lambda$ is independent of the choice of orbit, whenever
$\Gamma$ is non-elementary, and its complement $\Omega$ in $\P^1$ is
the largest discontinuity domain. Moreover, $\Omega$ coincides also
with the region of equicontinuity, admits  a complete metric and
volume which are invariant under the action of $\Gamma$, facts which
are essential for defining the Patterson-Sullivan measure for
Kleinian groups (see for instance \cite{Pat, Su}). When we go look
at discrete subgroups of $PSL(3, \mathbb C)$ acting on $\P^2$, there
are examples where the action on the complement of the set of
accumulation points of the orbits is neither discontinuous nor
equicontinuous. There are also examples (see \cite{Ku}) where there
is a largest discontinuity set but it does not agree with the
equicontinuity set. Hence it seems that there are different notions
of ``limit set" for discrete subgroups of $PSL(3,\C)$ acting on
$\Bbb{P}^2_\Bbb{C}$.

For discrete actions in general topological spaces there is a notion
of limit set due to Kulkarni, see \cite{Ku}, which has the important
property of granting  that the action on its complement is
discontinuous and in the case of conformal automorphisms of the
$n$-sphere $\Bbb{S}^n$ agrees with the usual notion of limit set. In
\cite{SV1, SV2} there are some interesting families of projective
groups and  their limit set in Kulkarni's sense. Later  in  the
articles  \cite {cano, Na1, Na2}, the authors provide  sufficient
conditions,  for particular subgroups $\Gamma$ of $PSL(3, \C)$, to
grant that the equicontinuity set and Kulkarni's discontinuity
region, $\Omega(\Gamma)$,  agree. It was also shown that, in such
particular case, $\Omega(\Gamma)$ is the largest discontinuity set.

One of the goals of this paper is the following:
\begin{theorem} \label{main1}
Let $\Gamma\subset PSL(3,\C)$  be a discrete group, then the
equicontinuity set of $\Gamma$ is a discontinuity region. Moreover,
if $ U$ is an open $\Gamma$ invariant subset with at least three
lines in general position lying on its complement, then $U$ is
contained in the equicontinuity set of $\Gamma$. Furthermore, if
Kulkarni's limit set $\Lambda(\Gamma)$ contains at least three lines
in general position, then Kulkarni's discontinuity region
$\Omega(\Gamma)= \P^2 \setminus \Lambda(\Gamma)$ is equal to the
equicontinuity set of the group $\Gamma$.
\end{theorem}

In section \ref{psprojective},we split the Theorem \ref{main1} in
three parts: Theorem \ref{gamactsdisconeq}, Theorem \ref{dissubeq}
and Theorem \ref{eqeqdis}. The proofs of these theorems rely on
{\it pseudo-projective} maps, which provide a compactification of
the non-compact Lie group $PSL(3,\C)$.


Kulkarni's limit set $\Lambda(\Gamma)$, for a  subgroup $\Gamma$ of
$PSL(3, \C)$ acting on $\P^2$, may seem a complicated object at
first sight. However, when $\Gamma =\langle \alpha \rangle$ is a
cyclic group, generated by the element $\alpha \in PSL(3, \C)$, then
the limit set of this cyclic group, denoted $\Lambda(\alpha)$, can
be found (see \cite{Na2}) and it is  one of the following (according
to the Jordan canonical form of $\alpha$):
\begin{itemize}
\item The empty set,
\item all of $\P^2$,
\item one single complex line,
\item one complex line and one point,
\item two distinct complex lines.
\end{itemize}
In section \ref{lambdaunlin}, we define the set
$C(\Gamma)=\overline{\cup _{\gamma \in \Gamma} \Lambda(\gamma)}$ and
we prove the following:

\begin{theorem}\label{main2}
Let $\Gamma\subset PSL(3,\Bbb{C})$ be a discrete group, if   the
number of complex lines in general position in $\Lambda(\Gamma)$ and
$C(\Gamma)$ is at least three, then
\[
\Lambda(\Gamma)=C(\Gamma)=\overline{\bigcup_{\gamma\in \Gamma}
\Lambda(\gamma)},
\]
and   $\Lambda(\Gamma)$ is the union of complex lines. Moreover the
discontinuity region according to Kulkarni, $\Omega(\Gamma)$, agrees
with the equicontinuity set. Furthermore, when  $\Gamma$ acts on
$\P^2$ without global fixed points, then $\Omega(\Gamma)$ is the
largest open set where $\Gamma$ acts properly and discontinuously.
\end{theorem}

If $\Gamma \subset PU(2,1)$ is a discrete subgroup, then $\Gamma$
acts on the complex hyperbolic plane $\H ^2_\C$ by isometries. In
this case, the limit set according to Chen-Greenberg of $\Gamma$ is
obtained as the set of accumulation points, in $\partial \H ^2 _\C$,
of the $\Gamma$-orbit of any point in $\H ^2 _\C$. Also, it can be
obtained as the closure of the set of fixed points of non-elliptic
elements in $\Gamma$. When we consider the action of $\Gamma$ in
$\P^2$, the Kulkarni's limit set $\Lambda(\Gamma)$ is given in the
following way (see \cite{Na1}): Let $F_{\Gamma}$ be the set of
points $p \in
\partial \H^2 _\C$ such that there exists a non-elliptic element
$\gamma \in \Gamma$ that satisfies  $\gamma p=p$. In what follows,
$(\P^2)^*$, denotes the space of complex lines in $\P^2$. If
$\mathcal{E}(\Gamma)$ denotes the subset of $(\P^2)^*$ consisting of
those complex lines tangent to $\partial \H ^2 _\C$ at points of
$F_{\Gamma}$, then
$$\Lambda(\Gamma) = \bigcup _{l \in \overline{\mathcal{E}(\Gamma)}} l=
\overline{\bigcup_{l \in \mathcal{E}(\Gamma)} l}.$$

Moreover, when $l\in \overline{\mathcal{E}(\Gamma)}$, the closure of
the $\Gamma$-orbit of the complex line  $l$ in  $(\P^2)^*$ is equal
to $\overline{\mathcal{E}(\Gamma)}$. In other words, the action of
$\Gamma$ in the set of lines $\overline{\mathcal{E}(\Gamma)}$ is
minimal.

Let $\Gamma \subset PU(2,1)$ be a discrete subgroup such that
$\Lambda(\Gamma)=\P^2 \setminus  \H ^2 _\C$, then there exist
complex lines contained in $\Lambda(\Gamma)$ not belonging to
$\overline{\mathcal{E}(\Gamma)}$, so that $\mathcal{E}(\Gamma)$
gives us a canonical choice of lines to describe the limit set
$\Lambda(\Gamma)$ as a union of lines. In section \ref{lambdaunlin},
we prove the following generalization of the main theorems in
\cite{Na1, BN}:

\begin{theorem} \label{main3}
Let  $\Gamma \subset PSL(3, \C)$ be an infinite discrete subgroup,
without fixed points nor invariant lines. Let $\mathcal{E}(\Gamma)$
be the subset of $(\P^2 )^*$ consisting of all the complex lines $l$
for which there exists an element $\gamma \in \Gamma$ such that $ l
\subset \Lambda(\gamma)$.
\begin{enumerate}
\item[a)] $Eq(\Gamma)= \Omega (\Gamma)$, is the maximal open set on which
$\Gamma$ acts properly and discontinuously. Moreover, if
$\mathcal{E}(\Gamma)$ contains more than three complex lines then
every connected component of $\Omega(\Gamma)$ is complete Kobayashi
hyperbolic (compare with \cite{BN}).

\item[b)] The set
$$\Lambda(\Gamma)= \overline{\bigcup _{ l \in \mathcal{E}(\Gamma)} l}
=\bigcup _{ l \in \overline{\mathcal{E}(\Gamma)}} l =
\overline{\bigcup _{\gamma \in \Gamma} \Lambda(\gamma)}$$ is
path-connected.

\item[c)]If $\mathcal{E}(\Gamma)$ contains more than three lines,
then $\overline{\mathcal{E}(\Gamma)} \subset (\P ^2)^*$ is a perfect
set. Also, it is the minimal closed $\Gamma$-invariant subset of
$(\P^2 )^*$.
\end{enumerate}
\end{theorem}



\section{Preliminaries and Notation.}

\subsection{Projective Geometry}
We recall that the complex projective plane
$\mathbb{P}^2_{\mathbb{C}}$ is defined as the quotient space
$$(\mathbb{C}^{3}\setminus  \{(0,0,0)\})/\mathbb{C}^*,$$
 where $\mathbb{C}^*= \C \setminus \{0\}$ acts on $\mathbb{C}^3\setminus \{(0,0,0)\}$ by the usual scalar
multiplication. This is a complex $2$-dimensional Riemannian
manifold which is compact, connected and is naturally equipped with
the Fubini-Study metric. Let $[\mbox{ }]:\mathbb{C}^{3}\setminus
\{(0,0,0)\}\rightarrow \mathbb{P}^{2}_{\mathbb{C}}$ be the quotient
map. If $\beta=\{e_1,e_2,e_3\}$ is the standard basis of
$\mathbb{C}^3$, we write $[e_j]=e_j$ and if $w=(w_1,w_2,w_3)\in
\mathbb{C}^3\setminus \{(0,0,0)\}$ then we write
$[w]=[w_1:w_2:w_3]$. Also, the set $l\subset
\mathbb{P}^2_{\mathbb{C}}$ is said to be a complex line if
$[l]^{-1}\cup \{(0,0,0)\}$ is a complex linear subspace of dimension
$2$. Given  two distinct points $p,q\in \mathbb{P}^2_{\mathbb{C}}$,
there is a unique complex line passing through $p$ and $q$, such
line is denoted by $\Flecha{p,q}$.

Any 2-dimensional complex vector subspace of $\C ^3$ can be
expressed as the set of points $(z_1, z_2, z_3)$ satisfying an
equation of the form $A z_1+ B z_2 + C z_3 =0$. Conversely, any
element $(A, B, C) \in \C^3 \setminus \{(0,0,0)\}$ defines, up to a
non-zero scalar multiple, the 2-dimensional complex vector subspace
of $\C ^3$ consisting of the points $(z_1, z_2, z_3)$ such that $A
z_1+ B z_2 + C z_3 =0$. In this way, the space of complex lines in
$\P^2$, denoted by $(\P^2)^*$, can be identified with $\P ^2$.

Consider the action of $\mathbb{Z}_{3}$ (the cubic roots of the
unity) on $SL(3,\mathbb{C})$ given by the usual scalar
multiplication, then
$PSL(3,\mathbb{C})=SL(3,\mathbb{C})/\mathbb{Z}_{3}$ is a Lie group
whose elements are called projective transformations.  Let $[\mbox{
}]:SL(3,\mathbb{C})\rightarrow PSL(3,\mathbb{C})$ be   the quotient
map, we say that $\tilde\gamma \in GL(3,\mathbb{C})$ is a lift of
$\gamma \in PSL(3,\mathbb{C})$ if there is a cubic root $\tau$ of
$(Det(\tilde{\gamma}))^{-1}$ such that   $[\tau
\widetilde\gamma]=\gamma$, also, we will use the notation
$(\gamma_{ij})$ to denote elements  in $SL(3,\Bbb{C})$. One can show
that $ PSL(3, \mathbb{C})$ is a Lie group  that acts  transitively,
effectively and by biholomorphisms on $\mathbb{P}^2_{\mathbb{C}}$ by
$[\gamma]([w])=[\gamma(w)]$, where $w\in \mathbb{C}^3\setminus
\{(0,0,0)\}$ and    $\gamma\in SL(3,\mathbb{C})$. Also it is well
known that projective transformations take complex lines into
complex lines. Moreover, if the complex line $l \in (\P^2)^*$ is
represented by $[A: B: C]$  and $\gamma =[\tilde{\gamma}] \in PSL(3,
\C)$, then $\gamma (l) \in (\P^2)$ is represented by $[(A, B, C)
\tilde{\gamma}^{-1}]$.

\subsection{Kulkarni's Limit Set for Subgroups of $PSL(3,\C)$}
The Kulkarni's limit set is defined for actions of groups on very
general topological spaces (see \cite{Ku}), but we restrict our
attention to subgroups of $PSL(3,\C)$ acting on $\P^2$.
 Let $\Gamma\subset   PSL(3,\mathbb{C})$ be a subgroup:

\begin{enumerate}
\item $L_0(\Gamma)$ is defined as the closure  of  the points in
$\mathbb{P}^2_{\mathbb{C}}$ with infinite isotropy group.

\item $L_1(\Gamma)$ is the closure of the set  of cluster points  of
the $\Gamma$-orbit of the point $z$,  where  $z$ runs  over
$\mathbb{P}^2_{\mathbb{C}}\setminus  L_0(\Gamma)$.

Recall that $q$ is a cluster point  for  the family of sets
$\{\gamma (K)\}_{ \gamma \in \Gamma}$, where $K \ne \varnothing$ is
a subset of $\mathbb{P}^2_{\mathbb{C}}$, if there is a sequence
$(k_m)_{m\in\mathbb{N}}\subset K$ and a sequence of distinct
elements $(\gamma_m)_{m\in\mathbb{N}}\subset \Gamma$ such that
$\gamma_m(k_m)\xymatrix{ \ar[r]_{m \rightarrow  \infty}&} q$.

\item  $L_2(\Gamma)$ is  the closure of cluster  points of the
family of compact sets $\{\gamma (K)\}_{ \gamma \in \Gamma}$,  where
$K$ runs  over all the compact subsets of
$\mathbb{P}^2_{\mathbb{C}}\setminus (L_0(\Gamma) \cup L_1(\Gamma))$.

\item The  \textit{Limit Set in the sense of Kulkarni} for $\Gamma$  is
defined as:  $$\Lambda (\Gamma) = L_0(\Gamma) \cup L_1(\Gamma) \cup
L_2(\Gamma).$$

\item The \textit{Discontinuity Region in the
sense of Kulkarni} of $\Gamma$ is defined as:
$$\Omega(\Gamma) = \mathbb{P}^2_{\mathbb{C}}\setminus
\Lambda(\Gamma).$$
\end{enumerate}

We say  that $\Gamma$ is a \textit{Complex  Kleinian Group} if
$\Omega(\Gamma)\neq \emptyset$, see \cite{SV1}. The following
proposition is obtained from \cite{Ku}.


\begin{proposition}
Let   $\Gamma\subset PSL(3,\Bbb{C})$  be a complex Kleinian group.
Then:

\begin{enumerate}
\item \label{i:pk1} $\Gamma$ is discrete and countable.

\item\label{i:pk2}
$\Lambda(\Gamma),\,L_0(\Gamma),\,L_1(\Gamma),\,L_2(\Gamma)$ are
closed $\Gamma$-invariant sets.

\item \label{i:pk3} $\Gamma$ acts
discontinuously on  $\Omega(\Gamma)$.
\end{enumerate}
\end{proposition}

As the reader can notice, the computation of the limit set
$\Lambda(\Gamma)$ can be very complicated. The following lemma (see
\cite{Na2}) helps us in the task of the computation of
$L_2(\Gamma)$.

\begin{lemma} \label{cuasiminimalidad}
Let $\Gamma$ be a subgroup of $PSL(3, \mathbb C)$. If $C \subset
\P^2 $ is a closed set such that for every compact subset $K \subset
\P^2  \setminus  C$, the cluster points of the family of compact
sets $\{ \gamma(K) \} _{\gamma \in \Gamma}$ are contained in
$L_0(\Gamma) \cup L_1(\Gamma)$, then $L_2(\Gamma) \subseteq C$.
\end{lemma}

In what follows, the limit set of the cyclic group $\langle \gamma
\rangle$ will be denoted $\Lambda(\gamma)$.

The non-trivial elements of $PSL(3,\C)$ can be classified as
elliptic, parabolic or loxodromic (see \cite{Na2}):

The \emph{elliptic} elements in $PSL(3,\C)$ are those elements
$\gamma$ that have a lift to $SL(3, \C)$ whose Jordan canonical form
is
\begin{displaymath}
\left(
\begin{array}{ccc}
e^{i\theta _1} & 0 & 0 \\
0 & e^{i \theta _2} & 0 \\
0 & 0 & e^{i \theta _3}
\end{array}
\right).
\end{displaymath}
The limit set $\Lambda(\gamma)$ for $\gamma$ elliptic is
$\varnothing$ or all of $\P^2$ according to whether the order of
$\gamma$ is finite or infinite. It is the appropriated moment to
remark that those subgroups of $PSL(3, \C)$ containing an elliptic
element of infinite order cannot be discrete.

The \emph{parabolic} elements in $PSL(3, \C)$ are those elements
$\gamma$ such that the limit set $\Lambda(\gamma)$ is equal to one
single complex line. If $\gamma$ is parabolic then it has a lift to
$SL(3,\C)$ whose Jordan canonical form is one of the following
matrices:
\begin{displaymath}
\left(
\begin{array}{ccc}
1 & 1 & 0 \\
0 & 1 & 0 \\
0 & 0 & 1
\end{array}
\right) , \left(
\begin{array}{ccc}
1 & 1 & 0 \\
0 & 1 & 1 \\
0 & 0 & 1
\end{array}
\right) , \left(
\begin{array}{ccc}
e^{2\pi i t} & 1 & 0 \\
0 & e^{2\pi i t} & 0 \\
0 & 0 & e^{-4\pi i t}
\end{array}
\right) , e^{2 \pi i t} \ne 1
\end{displaymath}
In the first case $\Lambda(\gamma)$ is the complex line consisting
of all the fixed points of $\gamma$, in the second case
$\Lambda(\gamma)$ is the unique $\gamma$-invariant complex line. In
the last case $\Lambda(\gamma)$ is the complex line determined by
the two fixed points of $\gamma$.

There are four kinds of \emph{loxodromic} elements in $PSL(3,\C)$:
\begin{itemize}
\item The \emph{complex homotheties} are those elements $\gamma \in PSL(3, \C)$ that
have a lift to $SL(3,\C)$ whose Jordan canonical form is
\begin{displaymath}
\left(
\begin{array}{ccc}
\lambda  & 0 & 0 \\
0 & \lambda  &0 \\
0 & 0 & \lambda ^{-2}
\end{array}
\right) , \quad |\lambda| \ne 1,
\end{displaymath}
and its limit set $\Lambda(\gamma)$ is the set of fixed points of
$\gamma$, consisting of a complex line and a point.
\item The \emph{screws} are those elements $\gamma \in PSL(3, \C)$
that have a lift to $SL(3, \C)$ whose Jordan canonical form is
\begin{displaymath}
\left(
\begin{array}{ccc}
\lambda  & 0 & 0 \\
0 & \mu  &0 \\
0 & 0 & (\lambda \mu) ^{-1}
\end{array}
\right) , \quad \lambda \ne \mu, \, |\lambda| =|\mu| \ne 1,
\end{displaymath}
and its limit set $\Lambda(\gamma)$ consists of the complex line,
$l$, on which $\gamma$ acts as an elliptic transformation of $PSL(2,
\C)$ and the fixed point of $\gamma$ not lying in $l$.
\item The \emph{loxoparabolic} elements $\gamma \in PSL(3, \C)$ have
a lift to $SL(3, \C)$ whose Jordan canonical form is
\begin{displaymath}
\left(
\begin{array}{ccc}
\lambda  & 1 & 0 \\
0 & \lambda  &0 \\
0 & 0 & \lambda ^{-2}
\end{array}
\right) , \quad |\lambda| \ne 1,
\end{displaymath}
and the limit set $\Lambda(\gamma)$ consists of two
$\gamma$-invariant complex lines. The element $\gamma$ acts on one
of these lines as a parabolic element of $PSL(2, \C)$ and on the
other as a loxodromic element of $PSL(2, \C)$.
\item The \emph{strongly loxodromic} elements $\gamma \in PSL(3, \C)$
have a lift to $SL(3, \C)$ whose Jordan canonical form is
\begin{displaymath}
\left(
\begin{array}{ccc}
\lambda _1  & 0 & 0 \\
0 & \lambda _2 &0 \\
0 & 0 & \lambda _3
\end{array}
\right) , \quad |\lambda_1| < |\lambda_2|< |\lambda_3|.
\end{displaymath}
This kind of transformation has three fixed points, one of them is
attracting, other is repelling and the last one is a saddle. The
limit set $\Lambda(\gamma)$ is equal to the union of the complex
line determined by the attracting and saddle points and the complex
line determined by the saddle and repelling points.
\end{itemize}

\begin{theorem}\label{polloteo} \cite{cano}
Let $\Gamma \subset PSL(3, \C)$ be a discrete and infinite group
then
\begin{enumerate}
\item there exists $\gamma _0 \in \Gamma$ such that $\gamma _0$ has infinite
order,
\item If $\Gamma$ acts properly and discontinuously on $U \subset \P^2$, then
at least one of the complex lines in $\Lambda(\gamma _0)$, is
contained in $\P^2 \setminus U$. In particular, there exists a
complex line $l_0$ such that $l_0 \subset \Lambda(\gamma _0)$ and
$l_0 \subset \Lambda(\Gamma)$.
\end{enumerate}
\end{theorem}


\section{ Pseudo-Projective Maps and Equicontinuity. }
\label{psprojective}
We denote by $M_{3 \times 3}(\C)$ the space of all $3 \times 3$
matrices with complex entries equipped with the standard topology.
The quotient space
\[ (M_{3 \times 3}(\C) \setminus \{\mathbf{0}\})/\C^* \]
is called the space of \emph{pseudo-projective maps of} $\P^2$ and
it is naturally identified with the projective complex space $\P^8$.
Since $GL(3,\C)$ is an open, dense, $\C^*$-invariant set of $M_{3
\times 3}(\C) \setminus \{\mathbf{0}\}$, we obtain that the space of
pseudo-projective maps of $\P^2$ is a compactification of $PSL(3,
\C)$. As in the case of projective maps, if $\mathbf{s} \in M_{3
\times 3}(\C) \setminus \{\mathbf{0}\}$, then $[\mathbf{s}]$ denotes
the equivalence class of the matrix $\mathbf{s}$ in the space of
pseudo-projective maps of $\P^2$. Also, we say that $\mathbf{s}\in
M_{3 \times 3}(\C)\setminus \{\mathbf{0}\}$ is a lift of the
pseudo-projective map $S$, whenever $[\mathbf{s}]=S$.

Let $S$  be an element in $( M_{3 \times 3}(\C)\setminus
\{\mathbf{0}\})/ \C ^*$ and $\mathbf{s}$ a lift to $M_{3 \times
3}(\C)\setminus \{\mathbf{0}\}$ of $S$. The matrix $\mathbf{s}$
induces a non-zero linear transformation $s:\mathbb
{C}^{3}\rightarrow \mathbb {C}^{3}$, which is not necessarily
invertible. Let $Ker(s) \subsetneq \C ^3$ be its kernel and let
$Ker(S)$ denote its projectivization to $\P^2$, taking into account
that $Ker(S):= \varnothing$ whenever $Ker(s)=\{(0,0,0)\}$.
 Then $S$ can be considered as a map
$$S:\mathbb {P}^{2}_\mathbb {C}\setminus Ker(S) \rightarrow \mathbb {P}^{2}_\mathbb {C}$$
$$ S([v])=[ s(v)];$$
this is well defined because $v\notin Ker(s)$. Moreover, the
commutative diagram below implies that $S$ is a holomorphic map:
$$
 \xymatrix{
\mathbb {C}^{3}\setminus Ker ( s) \ar[rr]^{s} \ar[d]_{[\, \, \,
\,]}& & \mathbb {C}^{3}\setminus \{(0,0,0)\}\ar[d]^{[\, \, \,
\,]}\\
\mathbb {P}^{2}_\mathbb {C}\setminus Ker(M) \ar[rr]_{S} && \mathbb
{P}^{2}_\mathbb {C} } \;
$$
The \emph{image} of $S$, denoted $Im(S)$, is defined as the subset
of $\P^2$ given by
   $$Im(S):=[s(\mathbb {C}^{3})\setminus\{(0,0,0)\}] \, ,$$
and we have that $$dim_\mathbb {C}(Ker(S))+dim_\mathbb {C}(
Im(S))=1\,.$$

\begin{definition} \label{defeq}
The {\it equicontinuity set} for a family $\mathcal{F}$ of
endomorphisms of $\mathbb {P}^2_\mathbb {C}$, denoted
$Eq(\mathcal{F})$, is defined to be the set of points $z\in
\mathbb{P}^2_\mathbb{C}$ for which there is an open neighborhood $U$
of  $z$   such that $\{ f\vert_U : f \in \mathcal{F}\}$ is a normal
family (where normal family means that every sequence of distinct
elements has a subsequence which converges uniformly on compact
sets).
\end{definition}

It is not hard to see that $Eq(\mathcal F)$ is an open set and, in
the particular case when the family $\mathcal F$ consists of the
elements of a group $\Gamma \subset PSL(3, \C)$, the equicontinuity
set $Eq(\Gamma)$ is $\Gamma$-invariant. The following lemma helps us
to relate the equicontinuity set of a discrete group $\Gamma \subset
PSL(3,\C)$ to the set of pseudo-projective maps obtained as limits
of the elements of $\Gamma$ .
\begin{lemma}\label{delasucesion}
Let  $(\gamma _n)$ be a sequence of elements in $PSL(3,
\mathbb{C})$, then there exists a subsequence, still denoted
$(\gamma_n)$, and a pseudo-projective map $S$ such that:

\begin{itemize}
\item[1)] The sequence $(\gamma_n)$ converges uniformly to $S$ on compact subsets of
$\mathbb{P}^2_{\mathbb{C}}\setminus Ker(S)$.

\item[2)] If $Im(S)$ is a complex line, then there exists a 
pseudo-projective map $T$ such that $\gamma ^{-1}_n \underset{n \to
\infty}{\longrightarrow} T$ uniformly on compact subsets of
$\mathbb{P}^2_{\mathbb{C}}\setminus Ker(T)$. Moreover, $Im(S)=
Ker(T)$ and $Ker(S)=Im(T)$.

\item[3)] If $Ker(S)$ is a complex line, then
there exists a pseudo-projective map $T$ such that $\gamma ^{-1}_n
\underset{n \to \infty}{\longrightarrow} T$ uniformly on compact
subsets of $\mathbb{P}^2_{\mathbb{C}}\setminus Ker(T)$ and $Im(S)
\subset Ker(T)$. Moreover, if $l$ is a complex line not passing
through $Im(S)$ then the sequence of complex lines $\gamma_n
^{-1}(l)$ goes to the complex line $Ker(S)$ as $n \to \infty$.

\end{itemize}
\end{lemma}
\begin{proof}

1) For every $n \in \mathbb N$, let $\mathbf g _n =(g_{ij}^{(n)})$
be a lift to $GL(3,\C)$ of $\gamma_n$. The unitary ball in
$Mat_{3\times 3}(\C)=\C ^9$ with respect to the norm $\| \mathbf g
\| _{\infty} = \underset{1 \leqslant i,j \leqslant
3}{\max}|g_{ij}|$, is a compact set. It follows that the sequence
$\tilde{\mathbf{g}}_n = \frac{\mathbf{g}_n}{\|\mathbf{g}_n
\|_{\infty}}$ has a subsequence, still denoted
$\tilde{\mathbf{g}}_n$, such that $\tilde{\mathbf{g}}_n \to
\mathbf{s} \in Mat _{3 \times 3}( \C)$, as $n \to \infty$. We remark
that $\mathbf{s} \ne \mathbf{0}$. Now, if $K$ is a compact subset of
$\P^2 \setminus Ker(S)$, we consider the sets $C = \{v \in \C
^3\setminus \{(0,0,0)\} : [v] \in K \}$ and $\tilde{K}=\{ v/\| v \|
\in \C ^3: v \in C \}$. We notice that $\tilde{K}$ is a compact
subset of $\C ^3\setminus \{(0,0,0) \}$ and $[\tilde{K}]=K$. Since
$\tilde{\mathbf{g}}_n \underset{n \to \infty}{\longrightarrow}
\mathbf{s}$ uniformly on $\tilde{K}$ we obtain that
$\gamma_n=[\tilde{\mathbf{g}}_n] \underset{n \to
\infty}{\longrightarrow} [\mathbf{s}]=S$ uniformly on $K$.

2) For every $n \in \mathbb N$, let $\mathbf g _n =(g_{ij}^{(n)})$
be a lift to $GL(3, \C)$ of $\gamma_n$. We can assume that $\mathbf
g _n \to \mathbf{s}$ as $n \to \infty$, where $\mathbf s$ is a lift
to $Mat _{3 \times 3}( \C)$ of the pseudo-projective map $S$. The
cofactor matrix of $\mathbf{g}_n$, denoted $cof(\mathbf{g}_n)$ is a
lift to $GL(3,\C)$ of $\gamma _n ^{-1}$. Since the computation of
the cofactor matrix is a continuous function of the matrix entries,
we have that $cof(\mathbf{g}_n) \to cof(\mathbf s)$ as $n \to
\infty$. It follows that  $\gamma_n ^{-1}$ converges to the
pseudo-projective map $[cof(\mathbf s)]:=T$ uniformly on compact
subsets of $\P^2 \setminus Ker(T)$.

The matrix $\mathbf{s}$ may have the following Jordan canonical
forms.

\begin{displaymath}
\left(
\begin{array}{ccc}
\lambda _1 & 0 & 0 \\
0 & \lambda _2 & 0 \\
0 & 0 & 0
\end{array}
\right) , \left(
\begin{array}{ccc}
\lambda _3 & 1 & 0 \\
0 & \lambda _3 & 0 \\
0 & 0 & 0
\end{array}
\right) , \lambda _j \ne 0 \textrm{ for } j=1,2,3.
\end{displaymath}

In both cases the matrix $cof(\mathbf{s})$ is equal, up to
conjugation, to
\begin{displaymath}
\left(
\begin{array}{ccc}
0 & 0 & 0 \\
0 & 0 & 0 \\
0 & 0 & \mu
\end{array}
\right) , \quad \mu \ne 0,
\end{displaymath}
and the proof of 2) follows from a straightforward computation.

3) By part 1) we can assume that $\gamma _n ^{-1}$ converges to a
pseudo-projective map $T$ uniformly on compact subsets of $\P^2
\setminus Ker(T)$. We proceed by contradiction and we assume that
$Im(S)$ is not contained in $Ker(T)$. By hypothesis, $Ker(S)$ is a
complex line, so that $Im(S)$ consists of one single point denoted
by $p$, and $p \notin Ker(T)$. Now, if $x \in \P^2 \setminus Ker(S)$
then $S(x)=p$, it follows that $\gamma _n(x) \underset{n \to
\infty}{\longrightarrow}S(x)=p$. Since $p \notin Ker(T)$, we can
assume that the compact set $\{ \gamma _n(x): n \in \mathbb{N} \}
\cup \{ p\}$ is contained in $\P^2\setminus Ker(T)$. Hence $x=
\gamma _n ^{-1}(\gamma _n (x)) \underset{n \to
\infty}{\longrightarrow} T(p)$, which contradicts the fact that $T$
is a function defined on $\P^2 \setminus Ker(T)$. Therefore, $Im(S)
\subset Ker(T)$.

For every $n \in \mathbb N$, let $\mathbf g _n =(g_{ij}^{(n)})$ be a
lift to $GL(3,\C)$ of $\gamma_n$. We can assume that $\mathbf g _n
\to \mathbf{s} \in Mat_{3 \times 3}$ as $n \to \infty$, and
$[\mathbf{s}]=S$. Without loss of generality we can assume that

\begin{displaymath}
\mathbf s = \left(
\begin{array}{ccc}
0 & 0 & 0 \\
0 & 0 & 0 \\
0 & 0 & 1
\end{array}
\right) .
\end{displaymath}
In this case $Ker(S)=\Flecha{e_1, e_2} $ is precisely the set of
points in $\P^2$ that satisfy the equation $0z_1+0z_2+1 z_3=0$, then
the line $Ker(S)$ can be identified with $[0:0:1]$. Now, if $l$ is a
complex line not passing through $Im(S)=[e_3]$ then $l$ can be
identified with $[A:B:C]$, $C \ne 0$.
It follows that the complex line $\gamma_n ^{-1}(l)$ is identified
with
$$[(A, B, C) (\mathbf g _n ^{-1})^{-1}]=[(A, B, C)\mathbf g _n ],$$
which implies that the sequence of complex lines $\gamma_n ^{-1}(l)$
converges to the complex line identified with
$$[(A, B, C) \mathbf g] = [0: 0: C]=[0:0:1],$$
and this line is $Ker(S)$.
\end{proof}

\begin{lemma}\label{delasuc2}
Let $( \gamma_n)_{n\in \Bbb{N}}\subset PSL(3,\mathbb {C})$ be a
sequence that converges uniformly to the pseudo-projective map $S$
on compact subsets of $\mathbb{P}^2_{\mathbb{C}}\setminus Ker(S)$.
Then
$$Eq\{  \gamma_n:n\in \mathbb {N}\}=\mathbb {P}^2_\mathbb {C}\setminus
 Ker(S).$$
\end{lemma}
\begin{proof}
 The inclusion
$\P^2 \setminus Ker(S)\subset Eq\{\gamma_n : n\in\mathbb{N} \}$ is
obtained from part 1) of Lemma \ref{delasucesion}.

Now, we split the proof in two cases according to whether $Ker(S)$
is a point or a complex line.

If $Ker(S)$ is a point, denoted $k$, then $Im(S)$ is a complex line,
and given two distinct points $p$ and $q$ in $Im(S)$, there exist
points $p'$ and $q'$ arbitrarily close to $k$ such that $S(p')=p$
and $S(q')=q$ implying that $S$ cannot be extended to a continuous
function in a neighborhood of $k$, which implies $k \notin
Eq\{\gamma _n : n \in \mathbb N \}$. Therefore $Eq\{\gamma _n : n
\in \mathbb N \} \subset \P^2\setminus Ker(S)$.

Now, let us assume that $Ker(S)$ is a line, and let $p$ be a point
in $Ker(S)\setminus {Im (S)}$, we must prove that $p\notin
Eq\{\gamma_n : n\in\mathbb{N} \}$. We choose a complex line $l_1$
different from $Ker(S)$ and passing through $p$. Since $(\P^2)^*$ is
a compact space, we can assume that there exists a complex line
$l_2$ such that $\gamma_n(l_1)$ converges to $l_2$ as $n \to
\infty$. Given that $\gamma _n \underset{n \to
\infty}{\longrightarrow} S$ and $Im(S)$ is a point denoted $q$, we
have that $q$ lies in $l_2$.

If $q'$ is any point in $l_2\setminus \{q\}$ then there exists a
sequence of points $x_n$ in  $l_1$ such that $ \gamma_n(x_n) \to q'$
as $n \to \infty$. Since $l_1$ is compact, we can assume that there
exists a point $x\in l_1$ such that $x_n \to x$ as $n\to\infty$. We
notice that $x \in Ker(S)$ because otherwise we can assume that
$\overline{\{x_n : n\in \mathbb{N}\}}$ is a compact subset of
$\mathbb{P}^2_{\mathbb{C}}\setminus Ker(S)$, then $\gamma_n(x_n) \to
q$ as $n \to \infty$, so that $q = q'$, which is absurd. Hence $x=
p$, then $x_n \to p$ and $\gamma_n(x_n)\to q' \ne q$, which implies
that $p\notin Eq\{\gamma_n : n\in\mathbb{N} \}$. It follows that
$Ker(S)\setminus Im(S) \subset Eq\{\gamma_n : n\in\mathbb{N} \}$,
since the equicontinuity set is an open set, we have that  $Ker(S)
\subset \P^2 \setminus Eq\{\gamma_n : n\in\mathbb{N} \}$. Therefore,
$\P^2 \setminus Ker(S)= Eq\{\gamma_n : n\in\mathbb{N} \}$.
\end{proof}

\begin{theorem} \label{gamactsdisconeq}
Let $\Gamma \subset PSL(3, \C)$ be a discrete group then $\Gamma$
acts properly and discontinuously on the open $\Gamma$-invariant set
$Eq(\Gamma)$.
\end{theorem}
\begin{proof}
By contradiction, let $K$ be a compact subset of $Eq(\Gamma)$ such
that there exists a sequence $(\gamma _n)$ of distinct elements of
$\Gamma$ that satisfy $\gamma_n (K) \cap K \ne \varnothing$. For
every $n \in \mathbb{N}$ there exists $k_n \in K$ such that $\gamma
_n (k_n) \in K$. Without loss of generality we can assume that $k_n
\underset{n \to \infty}{\longrightarrow} k \in K$ and $\gamma
_n(k_n) \underset{n \to \infty}{\longrightarrow} \kappa \in K$.
Without loss of generality, we can assume, by lemma
\ref{delasucesion} 1), that there exists a pseudo-projective map $S$
such that $\gamma _n \underset{n \to \infty}{\longrightarrow} S$
uniformly on compact subsets of $\P^2 \setminus Ker(S)$. Since $K
\subset Eq(\Gamma)$, the lemma \ref{delasuc2} implies that $K
\subset \P^2\setminus Ker(S)$. As\'i que $\gamma_n(k_n) \underset{n
\to \infty}{\longrightarrow} S(k) $ then $S(k)=\kappa \in K$ which
means that $\kappa \in Im(S)$.

The Lemma \ref{delasucesion} implies that there is a pseudo
projective map $T$ such that $Im(S) \subset Ker(T)$, and the Lemma
\ref{delasuc2} implies that $Ker(T) \subset \P^2 \setminus
Eq(\Gamma)$. Hence $\kappa \in Im(S) \subset Ker(T) \subset \P^2
\setminus Eq(\Gamma)$, which contradicts that $K \subset
Eq(\Gamma)$. Therefore $\Gamma$ acts properly and discontinuously on
$Eq(\Gamma)$.
\end{proof}

\begin{theorem}\label{dissubeq}
If $\Gamma \subset PSL(3, \C)$ is a discrete group, 
 and  $U$ is a $\Gamma$-invariant open subset of $\P^2$
such that $\P^2  \setminus U$ contains at least three complex lines
in general position, then $$U \subset Eq(\Gamma).$$
\end{theorem}

\begin{proof}
It suffices to prove that $Ker(S) \subset \P^2 \setminus U ,$
whenever $S$ is a pseudo-projective map which is the limit of a
sequence $(\gamma _n)_{n \in \mathbb N} \subset \Gamma$.

If $Ker(S)$ is a point, denoted $p$, then $Im(S)$ is a complex line.
Moreover, there exists a pseudo-projective map $T$ such that $\gamma
_n ^{-1}$ goes to $T$, as $n \to \infty$, uniformly on compact
subset of $\P^2 \setminus Ker(T)$, such that $Ker(S)= Im(T)$ and
$Im(S)=Ker(T)$. There is a complex line $l_1 \ne Ker(T)$ such that
$l_1 \subset \P^2 \setminus U$. It follows that $\gamma _n ^{-1}(x)$
goes to $p$ as $n \to \infty$ for every point $x \in l_1 \setminus
Ker(S)$, which implies that $p \in \P^2 \setminus U$ (because $U$ is
a $\Gamma$-invariant open set).

If $Ker(S)$ is a complex line, then there exists a complex line
$l_1$ not passing through $Im(S)$ and contained in $\P^2 \setminus
U$, because $\P^2 \setminus U$ contains at least three lines in
general position. By the Lemma \ref{delasucesion} part iii), $\gamma
_n ^{-1} (l_1)$ goes to $Ker(S)$, as $n \to \infty$, which implies
that $Ker(S) \subset \P^2 \setminus U$.
\end{proof}

\begin{theorem}\label{eqeqdis}
Let $\Gamma \subset PSL(3, \C)$ be a discrete subgroup such that
$\Lambda(\Gamma)$ contains at least three lines in general position,
then $$\Omega(\Gamma)= Eq(\Gamma).$$
\end{theorem}

\begin{proof}
By Theorem \ref{gamactsdisconeq} we have that $L_0(\Gamma) \cap
Eq(\Gamma)= \varnothing = L_1(\Gamma) \cap Eq(\Gamma)$. Now we prove
that $L_2(\Gamma) \cap Eq(\Gamma)= \varnothing$. Let us denote by
$C$ the complement of $Eq(\Gamma)$ in $\P^2$, and let $K$ be a
compact subset of  $\P^2 \setminus C = Eq(\Gamma)$. By Lemma
\ref{cuasiminimalidad}, it suffices to prove that the accumulation
points of the orbit of $K$ are contained in $L_0 (\Gamma) \cup
L_1(\Gamma)$. Let $x$ be one of such accumulation points, then there
exists a sequence, $k_n$, of points in $Eq(\Gamma)$ such that $k_n
\to k \in Eq(\Gamma)$ as $n \to \infty$, and there is a sequence
$\gamma _n$ of different elements in $\Gamma$, such that $\gamma _n
(k_n) \to x$, as $n \to \infty$. By the Lemma \ref{delasucesion} we
can assume that there exists a pseudo-projective map $S$ such that
$\gamma _n$ goes to $S$, as $n \to \infty$, uniformly on compact
subsets of $\P^2  \setminus Ker(S)$. We have two cases according to
whether $Ker(S)$ is a point or a complex line.

If $Ker(S)$ is a complex line, then the Lemma \ref{delasucesion}
iii) implies that $\{k_n : n \in \mathbb N\} \cup \{k\} \subset
Eq(\Gamma) \subset \P^2 \setminus Ker(S)$. Thus, $\gamma_n(k_n) \to
x$ implies that $\gamma _n(k) \to S(k)=x$, and given that $k \in
Eq(\Gamma) \subset \P^2 \setminus L_0 (\Gamma)$, we deduce that $x
\in L_1(\Gamma)$.

If $Ker(S)$ is a point then $(\{k_n\} \cup \{k\}) \cap Ker(S)=
\varnothing$ because $\{k_n\} \cup \{k\} \subset Eq(\Gamma)$, and
the proof follows as in the former case.

Therefore $L_2(\Gamma) \cap Eq(\Gamma)= \emptyset$, then $Eq(\Gamma)
\subset \Omega(\Gamma)$. The proof of the reverse inclusion follows
easily from the Theorem \ref{dissubeq}.
\end{proof}

\section{$\Lambda(\Gamma)$ is union of complex lines}
\label{lambdaunlin} In this section $\Gamma \subset PSL(3, \C)$ is a
Complex Kleinian Group, and $C(\Gamma)$ denotes the set
$$C(\Gamma) = \overline{ \bigcup _{\gamma \in \Gamma} \Lambda(\gamma)}$$


\begin{lemma}\label{CsubLambda}
If $\Lambda(\Gamma)$ has at least three complex lines in general
position, then $\Lambda(\gamma) \subset \Lambda(\Gamma)$ for every
$\gamma \in \Gamma$. In particular, $C(\Gamma)\subset
\Lambda(\Gamma)$.
\end{lemma}

\begin{proof}
Let $\gamma$ be an element in $\Gamma$.

There are three cases depending on whether $\gamma$ is elliptic,
parabolic or loxodromic (\cite{Na2})

(1) If $\gamma$ elliptic then $\Lambda(\gamma)=\varnothing$ and the
result follows immediately.

(2) If $\gamma$ es parabolic, then we have three subcases:

\quad (i) We can assume, without loss of generality, that $\gamma$
has a lift given by the matrix.
\begin{displaymath}
\left(
\begin{array}{ccc}
1 & 1 & 0 \\
0 & 1 & 0 \\
0 & 0 & 1
\end{array}
\right) ,
\end{displaymath}

then $\Lambda(\gamma)=L_0(\gamma) = \overleftrightarrow{e_1, e_2}
\subset \L_0(\Gamma) \subset \Lambda(\Gamma)$.

\quad (ii) We can assume, without loss of generality, that $\gamma$
has a lift $\tilde{\gamma}$ which has the form
\begin{displaymath}
\left(
\begin{array}{ccc}
1 & 1 & 0 \\
0 & 1 & 1 \\
0 & 0 & 1
\end{array}
\right) .
\end{displaymath}
 It follows that, for any $n \in \mathbb N$,
\begin{displaymath}
\tilde{\gamma} ^n =\left(
\begin{array}{ccc}
1 & n & \frac{n(n-1)}{2} \\
0 & 1 & n \\
0 & 0 & 1
\end{array}
\right) .
\end{displaymath}
If $\Lambda(\gamma)= \Flecha{e_1 , e_2}$ is not contained in
$\Lambda(\Gamma)$, then there exists $z \in \C$ such that $[z:1:0]
\in \Omega(\Gamma)$. Let $w$ be any complex number, then the
sequence $a_n(w):= [z:1: \frac{2(w-n)}{n(n-1)}]$, $n \in \mathbb N$
tends to $[z:1:0]$ as $n \to \infty$. Thus, for $N>0$ large enough
($N$ depends on $w$), we have that $(a_n(w)) _{n \geqslant N}
\subset \Omega(\Gamma)$ for all $n >N$. We note that
$$\gamma ^n (a_n)=[z+ w : \frac{2w-(n+1) }{n-1} : \frac{2(w-n)}{n(n-1)}]
\underset{n \to \infty}{\longrightarrow} [-z -w: 1:0].$$ Hence, $[-z
-w: 1:0] \in \Lambda(\Gamma)$, and taking $w=-2z$ we obtain a
contradiction. Therefore, $\Lambda(\gamma) \subset \Lambda(\Gamma)$.

\quad (iii) We can assume, without loss of generality, that $\gamma$
has a lift $\tilde{\gamma}$ of the form:
\begin{displaymath}
\tilde{\gamma}  = \left(
\begin{array}{ccc}
\lambda & 1 & 0 \\
0 & \lambda & 0 \\
0 & 0 & \lambda ^{-2}
\end{array}
\right) , \quad |\lambda|=1.
\end{displaymath}
We proceed by contradiction. Let's assume that $\Lambda(\gamma)=
\Flecha{e_1, e_3}$ is not contained in $\Lambda(\Gamma)$, then there
exists $z \in \C$ such that $[z:0:1] \in \Omega(\Gamma)$. We note
that $a_n:= [z:1/n^2 : 1] \underset{n \to \infty}{\longrightarrow}
[z:0:1]$, so that for $N>0$ large enough, we have that $(a_n) _{n
\geqslant N} \subset \Omega(\Gamma)$ for all $n
>N$. The sequence,
$$\gamma ^n (a_n)=[z \lambda ^n+ \frac{n \lambda ^{n-1}}{n^2} : \frac{\lambda ^n}{n^{2}} : \lambda ^{-2n}]
=[z \lambda ^{3n}+ \frac{\lambda ^{3n-1}}{n}: \frac{\lambda
^{3n}}{n^2}:1], \quad n >N,$$ has cluster point $[z:0:1]$ which
implies that $[z:0:1] \in \Lambda(\Gamma)$, a contradiction.

(3) If $\gamma$ is loxodromic, we have four subcases:

\quad (i) If $\gamma$ has a lift $\tilde{\gamma}$ whose normal
Jordan form is:
\begin{displaymath}
\left(
\begin{array}{ccc}
\lambda _1 & 0 & 0 \\
0 & \lambda _2 & 0 \\
0 & 0 &  \lambda _3
\end{array}
\right) , \quad |\lambda _1|<|\lambda _2|<|\lambda _3|,
\end{displaymath}
then we can assume that $\tilde \gamma$ is such matrix. It follows
that $\Lambda(\gamma)= \Flecha{e_1, e_2} \cup \Flecha{e_2, e_3}$,
and the lambda-lemma implies that
$$\Flecha{e_1,e_2} \subset \Lambda(\Gamma) \,  \textrm{ \'o } \,
\Flecha{e_2, e_3} \subset \Lambda(\Gamma).$$ Let's assume, without
loss of generality, that $ \Flecha{e_1,e_2} \subset
\Lambda(\Gamma)$, then there exists $l \subset \Lambda(\Gamma)$ such
that $l$ does not pass through the point $[1:0:0]$ (because
$\Lambda(\Gamma)$ contains at least three complex lines in general
position). We note that the sequence of complex lines $\gamma ^n
(l)$ goes to the complex line $\Flecha{e_2, e_3}$ as $n \to \infty$,
therefore $\Flecha{e_2, e_3}\subset \Lambda(\Gamma)$.

\quad (ii) If $\gamma$ has a lift $\tilde{\gamma}$ whose Jordan
normal form is
\begin{displaymath}
\left(
\begin{array}{ccc}
\lambda _1 & 0 & 0 \\
0 & \lambda _1 & 0 \\
0 & 0 &  \lambda _1 ^{-2}
\end{array}
\right) , \quad |\lambda _1| \ne 1,
\end{displaymath}
then we can assume that $\tilde \gamma$ is such matrix. In this case
$\Lambda (\gamma)=L_0(\gamma) \subset L_0(\Gamma) \subset
\Lambda(\Gamma)$.

\quad (iii) If $\gamma$ has a lift $\tilde{\gamma}$ whose Jordan
normal form is
\begin{displaymath}
\left(
\begin{array}{ccc}
\lambda _1 & 0 & 0 \\
0 & \lambda _2 & 0 \\
0 & 0 &  \lambda _3
\end{array}
\right) , \quad |\lambda _1| = |\lambda _2| < |\lambda _3|,
\end{displaymath}
then we can assume that $\tilde \gamma$ is such  matrix, so that
$$\Lambda(\gamma)= \Flecha{e_1, e_2} \cup \{[0:0:1]\}.$$
Let $l$ be a complex line contained in $\Lambda(\Gamma)$ not passing
through $[0:0:1]$ then $\gamma ^{-n} (l)$ is a sequence of complex
lines that tends to the complex line $\Flecha{e_1, e_2}$ as $n$
tends to $\infty$, which implies that $\Flecha{e_1, e_2} \subset
\Lambda(\Gamma)$. Moreover, $\gamma ^n(z) \underset{n \to
\infty}{\longrightarrow} [0:0:1]$ whenever  $z \in P^2 _\C \setminus
\Flecha{e_1, e_2}$. Therefore $[0:0:1] \in \Lambda(\Gamma)$.

\quad (iv) If $\gamma$ has a lift $\tilde{\gamma}$ whose Jordan
normal form is
\begin{displaymath}
\left(
\begin{array}{ccc}
\lambda  & 1 & 0 \\
0 & \lambda  & 0 \\
0 & 0 &  \lambda  ^{-2}
\end{array}
\right) , \quad |\lambda | < 1,
\end{displaymath}
then we can assume that $\tilde \gamma$ is such matrix, in this
case, $\Lambda(\gamma)=\Flecha{e_1, e_2} \cup \Flecha{e_1,e_3}$. We
have that
\begin{displaymath}
\tilde{\gamma} ^n= \left(
\begin{array}{ccc}
\lambda ^n  & n \lambda ^{n-1} & 0 \\
0 & \lambda ^n & 0 \\
0 & 0 &  \lambda  ^{-2n}
\end{array}
\right) ,  n \in \mathbb N .
\end{displaymath}
Let's assume that $\Flecha{e_1, e_2}$ is not contained in
$\Lambda(\Gamma)$, then there exists $z \in \C$ such that $[z:1:0]
\in \Omega(\Gamma)$. If $w$ is a non-zero complex number, then the
sequence $[wz:w: n\lambda ^{3n-1}]$ tends to $[z:1:0]$ as $n \to
\infty$. Moreover $$ \gamma ^n([z:1: \frac{n \lambda ^{3n-1}}{w}])
\underset{n \to \infty}{\longrightarrow} [w:0:1],
$$
which implies that the complex line $\Flecha{e_1, e_3}$ is contained
in $\Lambda(\Gamma)$.

Now, assuming $\Flecha{e_1, e_2} \subset \Lambda(\Gamma)$, let $l$
be a complex line different from $\Flecha{e_1, e_2}$, which is
contained in $\Lambda(\Gamma)$ and does not pass through $[0:0:1]$,
then $l$ has an equation of the form
$$Az_1 +Bz_2 +Cz_3 =0, \quad |A|^2+|B|^2 \ne 0, C \ne 0.$$
It is not hard to check that $\gamma ^n (l)$, $n \in \mathbb N$ has
an equation of the same kind, with coefficients $(A'(n), B'(n),
C'(n))$ given by the equation
\begin{displaymath}
\left(
\begin{array}{c}
A'(n) \\
B'(n)\\
C'(n)
\end{array}
\right) =\left( \left(
\begin{array}{ccc}
\lambda ^n  & n \lambda ^{n-1} & 0 \\
0 & \lambda ^n & 0 \\
0 & 0 &  \lambda  ^{-2n}
\end{array}
\right) ^{T} \right) ^{-1}  \left(
\begin{array}{c}
A \\
B\\
C
\end{array}
\right)
\end{displaymath}

Thus, $(A'(n), B'(n), C'(n))=(\lambda ^{-n} A, \lambda ^{-n}B -n
\lambda^{-n-1}A, \lambda ^{2n} C)$, in other words, $$(A'(n), B'(n),
C'(n))=( A, B -n \lambda^{-1}A, \lambda ^{3n} C),$$ it follows that
$\gamma ^n (l) \underset{n \to \infty}{\longrightarrow} \Flecha{e_1,
e_3}$, which implies that $\Flecha{e_1, e_3} \subset
\Lambda(\Gamma)$.

If $\Flecha{e_1, e_3} \subset \Lambda(\Gamma)$, let us take a line
$l$ contained in $\Lambda(\Gamma)$ not passing through $[0:0:1]$,
then $l$ has an equation of the form
$$Az_1 +Bz_2 +Cz_3 =0, \quad  C \ne 0.$$
It is not hard to check that $\gamma ^{-n} (l)$, $n \in \mathbb N$
has an equation of the same kind, whose coefficients $(A'(-n),
B'(-n), C'(-n))$ are given by the equation
\begin{displaymath}
\left(
\begin{array}{c}
A'(-n) \\
B'(-n)\\
C'(-n)
\end{array}
\right) = \left(
\begin{array}{ccc}
\lambda ^n  & 0 & 0 \\
n \lambda ^{n-1} & \lambda ^n & 0 \\
0 & 0 &  \lambda  ^{-2n}
\end{array}
\right)   \left(
\begin{array}{c}
A \\
B\\
C
\end{array}
\right) .
\end{displaymath}

Thus $(A'(-n), B'(-n), C'(-n))=(\lambda ^n A ,n \lambda ^{n-1}A +
\lambda ^n B, \lambda ^{-2n}C)$, equivalently, $$(A'(-n), B'(-n),
C'(-n))=(\lambda ^{3n} A ,n \lambda ^{3n-1}A + \lambda ^{3n} B,
C),$$ since $|\lambda|<1$ we have that $\gamma ^{-n}(l) \underset{n
\to \infty}{\longrightarrow} \Flecha{e_1, e_2}$, which implies that
$\Flecha{e_1, e_2} \subset \Lambda(\Gamma)$.
\end{proof}


\begin{proof}[Proof of Theorem \ref{main2}]

First, we notice that Theorem \ref{dissubeq} implies that $\P^2
\setminus C(\Gamma)\subset Eq(\Gamma)$, because $\P^2 \setminus
C(\Gamma)$ is a $\Gamma$-invariant open set and $C(\Gamma)$ contains
at least three complex lines in general position. It follows from
Theorem \ref{eqeqdis} that  $\P^2 \setminus C(\Gamma)\subset
Eq(\Gamma)= \Omega(\Gamma)$. Therefore $\Gamma$ acts properly and
discontinuously on $\P^2 \setminus C(\Gamma)$ and $\Lambda(\Gamma)
\subset C(\Gamma)$.

Given that $\Lambda(\Gamma)$ contains at least three complex lines
in general position, the Lemma \ref{CsubLambda} implies that
$C(\Gamma)\subset \Lambda(\Gamma)$.

In order to prove that $C(\Gamma)$ is a union of complex lines, it
suffices to check that for every $p \in \cup _{\gamma \in \Gamma}
\Lambda (\gamma)$, there exists a complex line contained in
$\Lambda(\Gamma)$ passing through $p$. Hence, let us assume that $p
\in \Lambda(\gamma)$ for some $\gamma \in \Gamma$. We need only to
consider the case when $\Lambda(\gamma)$ consists of one complex
line $l$ and one point $q \notin l$. If $p \in l$ then there is
nothing to prove. If $p=q$ then there exists a complex line $l_1
\subset \Lambda(\Gamma)$, $l_1 \ne l$, so when  $n \to \infty$,
$\gamma ^n(l_1)$ or $\gamma ^{-n}(l_1)$ goes to a complex line
passing through $p$ and contained $\Lambda(\Gamma)$.

Since $C(\Gamma)$ contains three complex lines in general position
and $\Gamma$ has no fixed point, there exist three complex lines
$l_1, l_2, l_3$ and three elements $\gamma _1, \gamma _2, \gamma _3
\in \Gamma$ such that $l_1 \subset \Lambda(\gamma _1)$, $l_2 \subset
\Lambda(\gamma _2)$, $l_3 \subset \Lambda(\gamma _3)$. Now, let $U$
be a $\Gamma$-invariant open subset of $\P^2$ on which $\Gamma$ acts
properly and discontinuously.

If $\gamma _1, \gamma _2, \gamma _3$ are pairwise different, then,
by Theorem \ref{polloteo}, we can assume $l_1, l_2, l_3\subset
\P^2\setminus U$. It follows from Theorems \ref{dissubeq} and
\ref{eqeqdis} that $U \subset Eq(\Gamma)=\Omega(\Gamma)$.

If $\gamma _1, \gamma _2, \gamma _3$ are not pairwise different,
then without loss of generality we can assume that $\gamma _1 \ne
\gamma _2 = \gamma _3$. By Theorem \ref{polloteo} we can assume that
$l_1 \subset P^2 \setminus U$ and at least one of the lines $l_2$ or
$l_3$ is contained in $\P^2 \setminus U$, therefore $\P^2 \setminus
U$ contains at least two complex lines in general position. If there
were not more than two complex lines in general position contained
in $\P^2 \setminus U$, then the intersection point of $l_1$ and
$l_2$ or the intersection point of $l_1$ and $l_3$ would be a
$\Gamma$-fixed point, and it cannot happen. Therefore $\P^2
\setminus U$ contains at least three lines in general position, and
it follows from Theorems \ref{dissubeq} and \ref{eqeqdis} that $U
\subset Eq(\Gamma)=\Omega(\Gamma)$.
\end{proof}

In what follows, we only consider groups $\Gamma \subset PSL(3, \C)$
whose action on $\P^2 $ have no fixed point nor invariant lines.

\begin{lemma} \label{3lineslema}
If $\Gamma \subset PSL(3, \C)$ is a infinite discrete group, having
no fixed points nor invariant lines then  $ C(\Gamma)\cap
\Lambda(\Gamma)$ contains at least three complex lines in general
position.
\end{lemma}

\begin{proof}
Given that $\Gamma \subset PSL(3, \C)$ is an infinite and discrete
group, by Theorem \ref{polloteo}, there exists an infinite order
element  $\gamma_0 \in \Gamma$, and one complex line, $l_0$, such
that $l_0 \subset \Lambda(\gamma _0) \subset C(\Gamma)$, and $l_0
\subset \Lambda(\Gamma)$. Since $l_0$ is not $\Gamma$-invariant,
there exists an element $\gamma _1 \in \Gamma$ such that $l_1
:=\gamma _1 (l_0)\ne l_0$ (we notice that $l_1 \subset
\Lambda(\gamma _1 \gamma _0 \gamma _1 ^{-1}) \subset C(\Gamma)$).
Let $q_0$ be the intersection point of $l_0$ and $l_1$. If $\gamma
(l_0)$ passes through $q_0$ for every $\gamma \in \Gamma$, then
$q_0$ is a $\Gamma$-fixed point, which cannot happen. Thus, there
exists $\gamma _2 \in \Gamma$ such that $l_2:=\gamma_2 (l_0)$, $l_1$
and $l_0$ are in general position, and we see that $l_2 \subset
\Lambda(\gamma _2 \gamma _0 \gamma _2 ^{-1} ) \subset C(\Gamma)$.
\end{proof}

\begin{example} Let $\alpha$ and $\beta$ in $PSL(3, \C)$ be the elements
induced, respectively, by  the matrices
\begin{displaymath}
A= \left(
\begin{array}{ccc}
\frac{1}{2} & 0 &0 \\
0 & 1 & 0 \\
0 & 0 & 2
\end{array}
\right), \quad \quad B= \left(
\begin{array}{ccc}
0 & 0 & 1 \\
1 & 0 & 0 \\
0 & 1 & 0
\end{array}
\right) .
\end{displaymath}

Let  $G$ be the group $\langle A , B \rangle \subset GL(3, \C)$, and
$\Gamma = \langle \alpha , \beta \rangle \subset PSL(3, \C)$, then
$\Gamma$ is a discrete group, it has no fixed points nor invariant
lines and $\Lambda(\Gamma)= C(\Gamma) = \Flecha{e_1, e_2} \cup
\Flecha{e_2, e_3} \cup \Flecha{e_3, e_1}$.
\end{example}

\begin{proof}
a) The group $G$ is a discrete, so $\Gamma$ is a discrete group. In
order to proof this statement, we use the norm in $\{ \mathbf{g} \in
GL(3, \C): \det \mathbf{g} = \pm 1\}$ given by $\|\mathbf{g}\|= \max
{|g_{i,j}|}$. We notice that each element in $G$ is given by one of
the following matrices:
\begin{displaymath}
\left(
\begin{array}{ccc}
2^{m_1} & 0 & 0 \\
0 & 2^{m_2} & 0 \\
0 & 0 & 2^{m_3}
\end{array}
\right), \quad \left(
\begin{array}{ccc}
 0 & 0 & 2^{m_1} \\
2^{m_2} & 0 & 0 \\
0 &  2^{m_3} & 0
\end{array}
\right), \quad
 \left(
\begin{array}{ccc}
0 & 2^{m_1} & 0 \\
0 & 0 & 2^{m_2} \\
2^{m_3} & 0 & 0
\end{array}
\right),
\end{displaymath}
where $m_1, m_2, m_3 \in \mathbb Z$ and $m_1+m_2+m_3 =0$, because
the set $\mathbf{H}$ of such matrices, is a group, and $A \in
\mathbf{H}$, $B \in \mathbf{H}$. Since $G \subset \mathbf{H}$, it
suffices to prove that $\mathbf{H}$ is a discrete group, and , for
this, we only need to check that, for every $k>0$, the set $\{
\mathbf{h} \in \mathbf{H} \, | \, \| \mathbf{h} \| < k \}$ is
finite. Let us assume that $\| h \| < k$ and, without loss of
generality, $m_1 = \max\{m_1, m_2 , m_3\}>0$, then
$$0 < m_1 < \log _2 k,$$
$$-\log _2 k < m_2 + m_3 < 2 \log _2 k.$$
We have two cases according to whether $m_2$ and $m_3$ are both
negative or $m_1 m_2 <0$. In any case, there are finitely many
values of $m_1, m_2, m_3$ under these conditions.

(b) The group $\Gamma$ has no fixed points nor invariant lines,
because every fixed point of $\alpha$ is not fixed by $\beta$, and
every $\alpha$-invariant complex line is not $\beta$-invariant.

(c) $\Lambda(\Gamma)= C(\Gamma) = \Flecha{e_1, e_2} \cup
\Flecha{e_2, e_3} \cup \Flecha{e_3, e_1}$.

The Lemma \ref{3lineslema} implies that $C(\Gamma) \cap
\Lambda(\Gamma)$ contains three complex lines in general position,
then Theorem \ref{main2} implies that $\Lambda(\Gamma)= C(\Gamma)$.
It is not hard to check that $\Lambda(\gamma)= \Lambda(\gamma ^3)$,
and $\gamma ^3$ is given by a diagonal matrix, it follows that
$\Lambda(\gamma ^3) \subset \Flecha{e_1, e_2} \cup \Flecha{e_2, e_3}
\cup \Flecha{e_3, e_1}$. Therefore $C(\Gamma) \subset \Flecha{e_1,
e_2} \cup \Flecha{e_2, e_3} \cup \Flecha{e_3, e_1}$, and it is not
hard to see that $\Flecha{e_1, e_2} \cup \Flecha{e_2, e_3} \cup
\Flecha{e_3, e_1} \subset C(\Gamma)$.
\end{proof}

Inspired by the last example we state the following Lemma.

\begin{lemma} \label{4lineslema}
If $\Gamma \subset PSL(3, \C)$ is an infinite discrete group without
fixed points nor invariant lines, and $C(\Gamma)$ contains more than
three complex lines (not necessarily in general position) then $
C(\Gamma) \cap \Lambda(\Gamma)$ has at least \emph{four} complex
lines in general position.
\end{lemma}

\begin{proof}
The proof of this Lemma is an extension of the proof of
\ref{3lineslema}, so there exist three complex lines $l_0, l_1, l_2$
(in general position), contained in the set $ C(\Gamma) \cap
\Lambda(\Gamma)$. Let $q_0$, $q_1$, $q_2$ be the intersection points
of the pairs of complex lines $l_0$ and $l_1$, $l_1$ and $l_2$,
$l_2$ and $l_0$, respectively. Now we proceed by contradiction, we
assume that there are not more than three lines (in general
position) contained in $C(\Gamma) \cap \Lambda(\Gamma)$, then for
every $\gamma \in \Gamma$, the complex lines $\gamma (l_j)$ for
$j=0,1,2$ necessarily pass through one of the points $q_0, q_1,
q_2$. We can suppose, without loss of generality, that there exists
one such line, $l_3$, distinct from $l_0$ and $l_1$, passing through
$q_0$. If there are no more lines of the form $\gamma (l_j),$
$j=0,1,2$, then $q_0$ is fixed by the whole group $\Gamma$ which
cannot happen. Therefore, there exists another complex line $l_4$
such that $l_4$ is equal to $\gamma (l_j)$ for some $\gamma \in
\Gamma$ and for some $j=0,1,2$,  
and $l_4$ passes through a point $q_1$ distinct from $q_0$ (because
otherwise $q_0$ is fixed by $\Gamma$), then $l_0, l_2, l_3, l_4$ are
in general position.
\end{proof}

\begin{lemma} \label{lineatractora} Let $\gamma \in PSL(3, \C)$ be any
element. Let us assume that $l_1$ is a complex line contained in
$\Lambda(\gamma)$, then for every complex line $l$ different from
$l_1$ (except, maybe, for a family of complex lines in a pencil of
complex lines), some of the sequences of \textbf{\emph{distinct}}
complex lines, $(\gamma ^n(l)) _{n \in \mathbb N}$ or $(\gamma
^{-n}(l))_{n \in \mathbb N}$ goes to $l_1$, as $n \to \infty$.
\end{lemma}

\begin{proof}
We split the proof in several cases according to the classification
of the elements in $PSL(3, \C)$ given in \cite{Na2}.

(a) In the first case, we can assume that $\gamma$ has a lift of the
form
\begin{displaymath}
\left(
\begin{array}{ccc}
1 & 1& 0 \\
0 & 1 & 0 \\
0 & 0 & 1
\end{array}
\right),
\end{displaymath}
then $\Lambda(\gamma)$ is the complex line $\Flecha{e_1, e_3}$,
which is represented, in $(\P^2 )^*$, by $[0:1:0]$. If $l$ is a
complex line represented in $(\P^2)^*$ by $[A:B:C]$, then $\gamma ^n
(l), n \in \mathbb N$ is represented in $(\P^2)^*$ by
\begin{displaymath}
[A(n): B(n): C(n)]:= [(A, B, C)\left(
\begin{array}{ccc}
1 & -n & 0 \\
0 & 1 & 0 \\
0 & 0 & 1
\end{array}
\right) ] = [A: -nA+B: C].
\end{displaymath}

We obtain that $[A(n): B(n): C(n)] \underset{n \to
\infty}{\longrightarrow} [0:1:0]$, whenever $A \ne 0$, which implies
that $\gamma ^n(l) \underset{n \to \infty}{\longrightarrow}
\Flecha{e_1, e_3}$, whenever $l$ is a complex line not passing
through the point $[1:0:0]\in \P^2$.

b) We can assume that $\gamma$ has a lift of the form:
\begin{displaymath}
\left(
\begin{array}{ccc}
\zeta & 1& 0 \\
0 & \zeta & 0 \\
0 & 0 & \zeta ^{-2}
\end{array}
\right), \quad |\zeta| =1,
\end{displaymath}
then $\Lambda(\gamma)$ is the complex line $\Flecha{e_1, e_3}$ which
is represented, in $(\P^2 )^*$, by $[0:1:0]$. If $l$ is a complex
line represented in $(\P^2 )^*$ by $[A:B:C]$, then $\gamma ^n (l), n
\in \mathbb N$ is represented in $(\P^2 )^*$ by
\begin{displaymath}
[A(n): B(n): C(n)]:= [(A, B, C)\left(
\begin{array}{ccc}
\zeta ^{-n} & -n \zeta ^{-(n+1)} & 0 \\
0 & \zeta ^{-n} & 0 \\
0 & 0 & \zeta^{2n}
\end{array}
\right)]
\end{displaymath}
Hence, $$[A(n): B(n): C(n)]= [A \zeta ^{-n}: -n \zeta ^{-(n+1)}A +
\zeta ^{-n}B:\zeta ^{2n} C ].$$

We obtain that $[A(n): B(n): C(n)] \underset{n \to
\infty}{\longrightarrow} [0:1:0]$, whenever $A \ne 0$, which means
that $\gamma ^n(l) \underset{n \to \infty}{\longrightarrow}
\Flecha{e_1, e_3}$, whenever $l$ is a complex line not passing
through the point $[1:0:0]\in \P^2$.

(c) We can assume that $\gamma$ has a lift of the form:
\begin{displaymath}
\left(
\begin{array}{ccc}
1 & 1& 0 \\
0 & 1 & 1 \\
0 & 0 & 1
\end{array}
\right).
\end{displaymath}
In this case, $\Lambda(\gamma)$ is the complex line $\Flecha{e_1,
e_2}$ which is represented, in $(\P^2)^*$, by $[0:0:1]$. If $l$ is a
complex line represented in $(\P^2)^*$ by $[A:B:C]$, then $\gamma ^n
(l), n \in \mathbb N$ is represented in $(\P^2)^*$ by
\begin{displaymath}
[A(n): B(n): C(n)]:= [(A, B, C)\left(
\begin{array}{ccc}
1 & -n & \frac{n(n+1)}{2} \\
0 & 1 & -n \\
0 & 0 & 1
\end{array}
\right)] .
\end{displaymath}
Hence, $$[A(n): B(n): C(n)]=[A:-nA+B: \frac{n(n+1)}{2}A-nB+C],$$and
it follows that $[A(n): B(n): C(n)] \underset{n \to
\infty}{\longrightarrow} [0:0:1]$, which means that $\gamma ^n(l)
\underset{n \to \infty}{\longrightarrow} \Flecha{e_1, e_2}$ for any
line $l$.

(d) We can assume $\gamma$ has a lift of the form
\begin{displaymath}
\left(
\begin{array}{ccc}
\zeta & 1& 0 \\
0 & \zeta & 0 \\
0 & 0 & \zeta ^{-2}
\end{array}
\right), \quad |\zeta| >1.
\end{displaymath}
In this case $\Lambda(\gamma)$ is equal to the union of the complex
lines $\Flecha{e_1, e_2}$ and $\Flecha{e_1, e_3}$ which are
represented in $(\P^2)^*$, by $[0:0:1]$ and $[0:1:0]$, respectively.
If $l$ is a complex line represented in $(\P^2)^*$ by $[A:B:C]$,
then $\gamma ^n (l), n \in \mathbb N$ is represented in $(\P^2)^*$
by
\begin{displaymath}
[A(n): B(n): C(n)]:= [(A, B, C)\left(
\begin{array}{ccc}
\zeta ^{-n} & -n \zeta ^{-(n+1)} & 0 \\
0 & \zeta ^{-n} & 0 \\
0 & 0 & \zeta^{2n}
\end{array}
\right) ]
\end{displaymath}
Hence,
\begin{eqnarray*}
[A(n): B(n): C(n)] & =  & [A \zeta ^{-n}: -n \zeta ^{-(n+1)}A +
\zeta ^{-n}B:\zeta ^{2n} C ] \\
& = & [A \zeta ^{-3n}: -n \zeta ^{-3n-1}A+ \zeta ^{-3n}B : C ].
\end{eqnarray*}
So that $[A(n): B(n): C(n)] \underset{n \to \infty}{\longrightarrow}
[0:0:1]$, whenever $C \ne 0$, which means that $\gamma ^n(l)
\underset{n \to \infty}{\longrightarrow} \Flecha{e_1, e_2}$,
whenever $l$ is a complex line not passing through the point
$[0:0:1]\in \P^2$.

Now, if $l$ is a complex line represented in $(\P^2)^*$ by
$[A:B:C]$, then $\gamma ^{-n} (l), n \in \mathbb N$ is represented
in $(\P^2)^*$ by
\begin{displaymath}
[A(-n): B(-n): C(-n)]:= [(A, B, C)\left(
\begin{array}{ccc}
\zeta ^{n} & n \zeta ^{n-1} & 0 \\
0 & \zeta ^{n} & 0 \\
0 & 0 & \zeta^{-2n}
\end{array}
\right) ]
\end{displaymath}
Hence,
\begin{eqnarray*}
[A(-n): B(-n): C(-n)] & =  & [\zeta ^{n}A : n \zeta ^{n-1}A +
\zeta ^{n}B: \zeta ^{-2n}C  ] \\
& = & [A : n \zeta ^{-1}A+ B : \zeta ^{-3n} C  ] \\
& = & [A /n : A \zeta ^{-1}+ B/n : \zeta ^{-3n} C /n].
\end{eqnarray*}
 Thus,  $[A(-n): B(-n):C(-n)] \underset{n \to \infty}{\longrightarrow} [0:1:0]$,
 whenever $A \ne 0$ or $B \ne 0$, which implies that $\gamma ^{-n}(l)
\underset{n \to \infty}{\longrightarrow} \Flecha{e_1, e_3}$,
whenever $l$ is a complex line different from $\Flecha{e_1, e_2}$.

The remaining cases are proved in a similar way.
\end{proof}

We define the set $\mathcal{E}(\Gamma)$ as the subset of $(\P^2 )^*$
consisting of all the complex lines $l$ for which there exists an
element $\gamma \in \Gamma$ such that $ l \subset \Lambda(\gamma)$.
Also we define $E(\Gamma)$ as the subset of $\P^2$ given by
$E(\Gamma) = \overline{\bigcup _{l \in \mathcal{E}(\Gamma)} l}$. It
is not hard to see that $E(\Gamma) \subset C(\Gamma)$ and
$E(\Gamma)=\bigcup _{l \in \overline{\mathcal{E}(\Gamma)}} l$.

\begin{remark}
If $\mathcal{E}(\Gamma)$ contains at least two distinct complex
lines, then
$$C(\Gamma) = E(\Gamma)=\overline{\bigcup _{l \in \mathcal{E}(\Gamma)} l}
=\bigcup _{l \in \overline{\mathcal{E}(\Gamma)}} l .$$
\end{remark}
\begin{proof}
We need only to consider the case when $\Lambda(\gamma)$ has the
form $l \cup\{p\}$, where $l$ is a complex line and $p$ is a point.
By hypothesis, there exists a complex line $l_1 \in
\mathcal{E}(\Gamma)$ such that $l_1 \ne l$. If $l_1$ passes through
$p$ then $p \in E (\Gamma)$, so we can assume that $l_1$ does not
pass through $p$ and it follows that one of the sequences of lines
$\gamma ^n (l_1)$ or $\gamma ^{-n}(l_1)$ goes to a complex line in
$\overline{\mathcal{E}(\Gamma)}$ passing through $p$, as $n \to
\infty$. Therefore $p \in E(\Gamma)$ and $C(\Gamma) \subset
E(\Gamma)$.
\end{proof}

\begin{lemma}\label{Esperfecto}
If $\Gamma \subset PSL(3, \C)$ is a discrete group and
$\mathcal{E}(\Gamma)$ contains at least four complex lines in
general position, then $\overline{\mathcal{E}(\Gamma)} \subset
(\P^2)^*$ is a perfect set. Hence, $C(\Gamma)$ is a non-numerable
union of complex lines.
\end{lemma}

\begin{proof}
It suffices to prove that that each complex line in
$\overline{\mathcal{E}(\Gamma)}$ is an accumulation line of lines
lying in $\mathcal{E}(\Gamma)$. Furthermore, it is sufficient to
prove that each complex line in $\mathcal{E}(\Gamma)$ is an
accumulation line of lines lying in $\mathcal{E}(\Gamma)$. Let $l_1$
be a complex line in $\mathcal{E}(\Gamma)$, then there exists
$\gamma \in \Gamma$ such that $l_1 \subset \Lambda(\gamma)$. By the
Lemma \ref{lineatractora} and given that $\mathcal{E}(\Gamma)$
contains at least four lines in general position, we have that there
exists a line $l$ in $\mathcal{E}(\Gamma)$ such that some of the
sequences of distinct lines in $\mathcal{E}(\Gamma)$, $(\gamma
^{n}(l)) _{n \in \mathbb N }$ or $(\gamma ^{-n}(l)) _{n \in \mathbb
N }$ goes to $l_1$ as $n \to \infty$.
\end{proof}

\begin{proof}[Proof of Theorem \ref{main3}]

a) By Lemma \ref{3lineslema}, the set $\Lambda(\Gamma) \cap
C(\Gamma)$ contains at least three complex lines in general
position, then by Theorem \ref{eqeqdis} we have that
$Eq(\Gamma)=\Omega(\Gamma)$.

If $U \subset \P^2$ is a $\Gamma$-invariant open set on which
$\Gamma$ acts properly and discontinuously, then, by  Theorem
\ref{polloteo}, there exists a complex line $l$ contained in $\P^2
\setminus U$. Given that $\Gamma$ does not have invariant lines nor
fixed points, then $\P^2 \setminus U$ contains at least three
complex lines in general
position. 
Hence, by Theorem \ref{dissubeq}, $U \subset
Eq(\Gamma)=\Omega(\Gamma)$. Therefore $\Omega(\Gamma)$ is the
maximal open set on which $\Gamma$ acts properly and
discontinuously. The proof that  every connected component of
$\Omega(\Gamma)$ is complete Kobayashi hyperbolic is obtained
imitating the proof of Lemma 2.3 in \cite{BN}, and of the main
Theorem in \cite{BN}.

(b) The equality $\Lambda(\Gamma)= C (\Gamma)$ follows from the
Lemma \ref{3lineslema} and the Theorem \ref{main2}. Conexity follows
from the facts that two distinct complex lines always intersect and
a complex line is path-connected.

(c) The Lemma \ref{Esperfecto} implies that
$\overline{\mathcal{E}(\Gamma)}$ is a perfect set. Now, if
$\mathcal{D} \subset (\P^2)^*$ is a non-empty, closed
$\Gamma$-invariant subset, then there is a complex line $l \in
\mathcal{D}$ and the set $\{ \gamma (l) \, | \, \gamma \in \Gamma
\}$ contains at least three complex lines in general position. It
cannot happen that $\{ \gamma (l) \, | \, \gamma \in \Gamma \}$
contains only three complex lines because in such case
$\Lambda(\Gamma)$ would consist of only three complex lines, a
contradiction. Therefore, there are more than three complex lines in
the set $\{ \gamma (l) \, | \, \gamma \in \Gamma \}$ and imitating
the proof of the Lemma \ref{4lineslema}, we obtain that $\{ \gamma
(l) \, | \, \gamma \in \Gamma \}$ contains four lines in general
position. Therefore, $\mathcal{D}$ contains four lines in general
position. Let $\gamma$ be an element in $\Gamma$, applying the lemma
\ref{lineatractora} we deduce that every complex line contained in
$\Lambda(\gamma)$ is contained in $\mathcal{D}$, then
$\overline{\mathcal{E}(\Gamma)} \subset \mathcal{D}$.
\end{proof}

We would like to thank Professor Jos\'e Seade for enlightening
discussions during the process of this work and the kind hospitality
received by the three authors at IMATE-UNAM -Cuernavaca, and the
Facultad de Matem\'aticas, UADY. Also, during this time, the second
author was in a postdoctoral year at IMPA, and he is grateful to
this institution and its people, for their support and hospitality.

\end{document}